\documentclass[11pt]{amsart}%{article}

\usepackage[pdftex]{color,graphicx}
\usepackage[utf8]{inputenc}
\usepackage[english]{babel}
\usepackage{graphics}
\usepackage{enumerate}
\usepackage{array}
\usepackage{verbatim}
\usepackage{hyperref}
\usepackage{caption}
\usepackage{tikz-cd}
\usepackage{tikz}

\usepackage{multicol}
\usepackage{bbold}

\usetikzlibrary{graphs,decorations.pathmorphing,decorations.markings}

\usepackage{pdflscape}
\usepackage{rotating}
\usepackage{graphicx}

\usepackage{amsmath}
\usepackage{amsfonts}
\usepackage{amssymb}
\usepackage{amsthm}
\usepackage{mathrsfs}
\usepackage{mathtext}
\usepackage{mathtools}
\usepackage[mathscr]{eucal}

\usepackage[all,knot]{xy}

\newcommand{\id}[1]{\mathrm{id}_{#1}}

\newcommand{\End}[2]{\mathrm{End}_{#1}({#2})}

\newcommand{\cC}{\mathcal{C}}
\newcommand{\cD}{\mathcal{D}}

\newcommand{\msf}[1]{\mathsf{#1}}

\newcommand{\nc}{\newcommand}
\nc{\Db}{\mathrm{D}^{\mathrm{b}}}

\theoremstyle{plain}
\newtheorem{theo}[equation]{Theorem}

\theoremstyle{definition}

\theoremstyle{remark}

\title[Symmetric monoidal equivalences of TQFTs in dimension two and {F}robenius algebras]{Symmetric monoidal equivalences of topological quantum field theories in dimension two and {F}robenius algebras}

\author[P.\ S.\ Ocal]{Pablo S.\ Ocal}
\address{UCLA Mathematics Department, Los Angeles, CA 90095-1555, USA}
\email{socal@math.ucla.edu}
\urladdr{https://pabloocal.github.io/}

\date{July 2023}



\thanks{The author was supported by an AMS-Simons Travel Grant.}

\begin{document}

\begin{abstract}
We show that the canonical equivalences of categories between 2-dimensional (unoriented) topological quantum field theories valued in a symmetric monoidal category and (extended) commutative Frobenius algebras in that symmetric monoidal category are symmetric monoidal equivalences. As an application, we recover that the invariant of 2-dimensional manifolds given by the product of (extended) commutative Frobenius algebras in a symmetric tensor category is the multiplication of the invariants given by each of the algebras.
\end{abstract}

\maketitle

\section{Introduction}

There abound references formalizing the folklore statement that 2-dimensional topological quantum field theories (2-d TQFT) with values in a symmetric monoidal category $\cC$ and commutative Frobenius algebras in $\cC$ are equivalent, see for example \cite{abrams1996tqft, kock2004tqft}. A similar bijection between unoriented 2-dimensional topological quantum field theories with values in a symmetric monoidal category $\cC$ and extended commutative Frobenius algebras in $\cC$ was given in \cite{turaevturner2006utqft, czenky2023unoriented}. Here we fill a gap in the existing literature by showing that these correspondences are in fact symmetric monoidal equivalences. We also provide a cute application by recovering the multiplicativity of these invariants when the TQFTs take values in a symmetric tensor category, and remark how results of analogous generality for 2-dimensional homotopy quantum field theories (2-d HQFT) cannot be obtained.

\section{Preliminaries}

All the categories in this note will be \emph{symmetric monoidal} as in \cite[Definition 8.1.12]{EGNO2015tensor}, and using the coherence theorem for symmetric monoidal categories we regard all associators and unitors as identities. A \emph{symmetric monoidal equivalence} of symmetric monoidal categories is a \emph{symmetric monoidal functor} which is also an equivalence of categories (see \cite[Definitions 1.1.4, 1.1.5, 1.1.6, and Proposition 1.1.9]{windelborn2023hqft}, this is the particular case of a braided monoidal equivalence \cite[Definition 8.1.7]{EGNO2015tensor} between categories having a symmetric braiding). Given categories $\cC$ and $\cD$, then there is a category $\msf{SymMonCat}(\cC,\cD)$ of symmetric monoidal functors and symmetric monoidal natural transformations between them \cite[3.2.51]{kock2004tqft} (compare \cite[Definition 1.1.6]{windelborn2023hqft} with \cite[Definition 2.4.8 and Remark 2.4.9]{EGNO2015tensor}). Given $F$ and $G$ in $\msf{SymMonCat}(\cC,\cD)$ then $F\otimes G:\cC\to \cD$ given by $A\mapsto F(A)\otimes G(A)$ on objects of $\cC$ and $h\mapsto F(h)\otimes G(h)$ on morphisms of $\cC$ is a symmetric monoidal functor. The monoidal unit is the functor sending every object to $\mathbb{1}_{\cD}$ and every morphism to $\id{\mathbb{1}_{\cD}}$. The associator, unitors, and braiding are inherited from $\cD$. Altogether, this gives $\msf{SymMonCat}(\cC,\cD)$ a symmetric monoidal structure. The categories $\msf{2Cob}$ and $\msf{2UCob}$ of 2-dimensional oriented and unoriented cobordisms are symmetric monoidal, with monoidal structure given by disjoint union (see \cite[3.2.44]{kock2004tqft} and \cite[Definition 2.5]{czenky2023unoriented}). A \emph{2-d TQFT} \cite{atiyah1988tqft} with values in $\cC$ is a symmetric monoidal functor from $\msf{2Cob}$ to $\cC$. These form the category $\msf{SymMonCat}(\msf{2Cob},\cC)$. An \emph{unoriented 2-d TQFT} \cite{turaevturner2006utqft} with values in $\cC$ is a symmetric monoidal functor from $\msf{2UCob}$ to $\cC$. These form the category $\msf{SymMonCat}(\msf{2UCob},\cC)$.

We will be using the conventions in \cite[Definitions 7.8.1 and 7.20.3]{EGNO2015tensor} for \emph{algebra}, \emph{coalgebra}, and \emph{Frobenius algebra} in a category $\cC$. Given $A$ and $B$ Frobenius algebras in $\cC$, a \emph{morphism of Frobenius algebras} in $\cC$ is a morphism $f:A\to B$ in $\cC$ making the following diagrams commute (compare with \cite[2.4.4]{kock2004tqft}).
\begin{equation*}
\begin{tikzcd}
\mathbb{1} \arrow[swap]{d}{u_A} \arrow{r}{\id{\mathbb{1}}} & \mathbb{1} \arrow{d}{u_B}\\
A \arrow{r}{f} & B
\end{tikzcd}\quad
\begin{tikzcd}
A\otimes A \arrow[swap]{d}{m_A} \arrow{r}{f\otimes f} & B\otimes B \arrow{d}{m_B}\\
A \arrow{r}{f} & B
\end{tikzcd}\quad
\begin{tikzcd}
\mathbb{1} \arrow{r}{\id{\mathbb{1}}} & \mathbb{1}\\
A \arrow{r}{f} \arrow{u}{\epsilon_A} & B \arrow[swap]{u}{\epsilon_B}
\end{tikzcd}\quad
\begin{tikzcd}
A\otimes A \arrow{r}{f\otimes f} & B\otimes B\\
A \arrow{r}{f} \arrow{u}{\Delta_A} & B \arrow[swap]{u}{\Delta_B}
\end{tikzcd}
\end{equation*}
An \emph{extended Frobenius algebra} in $\cC$ is a tuple $(A,m,u,\Delta,\epsilon,\phi,\theta)$ where $(A,m,u,\Delta,\epsilon)$ is a Frobenius algebra in $\cC$, and $\phi:A\to A$ and $\theta:\mathbb{1}\to A$ are morphisms of Frobenius algebras in $\cC$ making the following diagrams commute (see \cite[Definition 2.10]{czenky2023unoriented}).
\begin{equation*}
\begin{tikzcd}
A \arrow{r}{\phi} \arrow[swap]{dr}{\id{A}} & A \arrow{d}{\phi}\\
 & A
\end{tikzcd}\quad
\begin{tikzcd}
A\otimes A \arrow{d}{m} & \mathbb{1}\otimes A \arrow{r}{\theta\otimes\id{A}} \arrow[swap]{l}{\theta\otimes\id{A}} & A\otimes A \arrow[swap]{d}{m}\\
A \arrow{rr}{\phi} & & A
\end{tikzcd}\quad
\begin{tikzcd}[column sep=20pt]
\mathbb{1}\otimes \mathbb{1} \arrow{d}{\theta\otimes\theta} \arrow{r}{l_{\mathbb{1}}} & \mathbb{1} \arrow{rr}{u} & & A \arrow{r}{\Delta} & A\otimes A \arrow[swap]{d}{\phi\otimes \id{A}} \\
A\otimes A \arrow{rr}{m} & & A & & A\otimes A \arrow[swap]{ll}{m}
\end{tikzcd}
\end{equation*}
Given $A$ and $B$ extended Frobenius algebras in $\cC$, a \emph{morphism of extended Frobenius algebras} in $\cC$ is a morphism of Frobenius algebras $f:A\to B$ in $\cC$ making the following diagrams commute (compare with \cite[Definition 2.5]{turaevturner2006utqft}).
\begin{equation*}
\begin{tikzcd}
\mathbb{1} \arrow[swap]{d}{\theta_A} \arrow{r}{\id{\mathbb{1}}} & \mathbb{1} \arrow{d}{\theta_B}\\
A \arrow{r}{f} & B
\end{tikzcd}\quad
\begin{tikzcd}
A \arrow[swap]{d}{\phi_A} \arrow{r}{f} & B \arrow{d}{\phi_B}\\
A \arrow{r}{f} & B
\end{tikzcd}
\begin{tikzcd}
\end{tikzcd}
\end{equation*}
There is a category $\msf{cFrob}(\cC)$ of commutative Frobenius algebras and morphisms of Frobenius algebras in $\cC$, and a category $\msf{cExtFrob}(\cC)$ of commutative extended Frobenius algebras and morphisms of extended Frobenius algebras in $\cC$. Given $A$ and $B$ in $\msf{cFrob}(\cC)$ then, by tedious diagram completion, the object $A\otimes B$ with multiplication $(m_A \otimes m_B)(\id{A}\otimes c_{A,B}\otimes \id{B})$, unit $(u_A\otimes u_B)(l^{-1}_{\mathbb{1}_{\cC}})$, comultiplication $(\id{A}\otimes c^{-1}_{A,B}\otimes \id{B})(\Delta_A\otimes \Delta_B)$, and counit $(l_{\mathbb{1}_{\cC}})(\epsilon_A\otimes \epsilon_B)$ is a commutative Frobenius algebra in $\cC$. When $A$ and $B$ are also in $\msf{cExtFrob}(\cC)$ then a similar procedure shows that the involution $\phi_A\otimes \phi_B$ and the distinguished element $(\theta_A\otimes \theta_B)(l^{-1}_{\mathbb{1}_{\cC}})$ make the above $A\otimes B$ into an extended commutative Frobenius algebra. Thus, the category $\cC$ induces a symmetric monoidal structure on the categories $\msf{cFrob}(\cC)$ and $\msf{cExtFrob}(\cC)$. We refer to \cite{ocaloswald2023dichotomy} for similar techniques and for Frobenius algebras obtained by replacing the braiding $c_{A,B}$ with more general isomorphisms in $\cC$.

\section{From correspondences to symmetric monoidal equivalences}

\begin{theo}\label{theo:oriented}
The canonical equivalence $\Phi:\msf{SymMonCat}(\msf{2Cob},\cC) \simeq \msf{cFrob}(\cC)$ is a symmetric monoidal equivalence.
\end{theo}

\begin{proof}
The assignment $\Phi$ mapping a TQFT to its evaluation at the circle $F \mapsto F(\mathbb{S}^1)$ and mapping a symmetric monoidal natural transformation to its component at the circle $(\eta:F\Rightarrow G) \mapsto (\eta_{\mathbb{S}^1}:F(\mathbb{S}^1)\to G(\mathbb{S}^1))$, is the canonical equivalence by \cite[Theorem 3]{abrams1996tqft} and \cite[Theorems 3.3.2 and 3.6.19]{kock2004tqft}. The unit $\mathbb{1}_{2d\,TQFT}:\msf{2Cob}\to \cC$ is mapped to the unit $\mathbb{1}_{\cC}$, the tensor product of TQFTs $F\otimes G$ is mapped to the tensor product of their respective evaluations $F(\mathbb{S}^1)\otimes G(\mathbb{S}^1)$, whence taking $J_{F,G} = \id{F(\mathbb{S}^1)\otimes G(\mathbb{S}^1)}$ makes $\Phi$ into a symmetric monoidal functor.
\end{proof}

\begin{theo}\label{theo:unoriented}
There is a canonical symmetric monoidal equivalence $\Phi:\msf{SymMonCat}(\msf{2UCob},\cC) \simeq \msf{cExtFrob}(\cC)$.
\end{theo}

\begin{proof}
There is a bijective correspondence between isomorphism classes of functors $F:\msf{2UCob}\to \cC$ and commutative extended Frobenius algebras in $\cC$ by \cite[Proposition 2.9]{turaevturner2006utqft} and \cite[Proposition 12]{czenky2023unoriented}. Let $\Phi$ be the same assignment as in the oriented case, so an unoriented TQFT $F$ is mapped to a commutative extended Frobenius algebra $F(\mathbb{S}^1)$, and a symmetric monoidal natural transformation $\eta:F\Rightarrow G$ is mapped to $\eta_{\mathbb{S}^1}:F(\mathbb{S}^1)\to G(\mathbb{S}^1)$ a morphism in $\cC$. Given $A$ a commutative extended Frobenius algebra in $\cC$ then up to isomorphism there is a unique symmetric monoidal functor $F:\msf{2UCob}\to \cC$ such that $F(\mathbb{S}^1) = A$, whence $\Phi$ is canonical. Following the reasoning in \cite[Section 3.3]{kock2004tqft}, the naturality of $\eta$ with respect to both pairs of pants, cup, and cap, implies that $\eta_{\mathbb{S}^1}$ is a morphism of Frobenius algebras in $\cC$. This naturality with respect to the orientation reversing cylinder implies that $\eta_{\mathbb{S}^1} \phi_{F(\mathbb{S}^1)} = \phi_{G(\mathbb{S}^1)} \eta_{\mathbb{S}^1}$, and naturality with respect to the punctured projective sphere seen as a morphism between the empty manifold and the circle implies that $\eta_{\mathbb{S}^1} \theta_{F(\mathbb{S}^1)} = \theta_{G(\mathbb{S}^1)}$. Thus $\eta_{\mathbb{S}^1}$ is a morphism of extended Frobenius algebras in $\cC$, and $\Phi$ is an equivalence of categories as in the oriented case. Since the monoidal structure of $\msf{SymMonCat}(\msf{2UCob},\cC)$ is inherited from the monoidal structure of $\msf{SymMonCat}(\msf{2Cob},\cC)$, the same reasoning as in Theorem~\ref{theo:oriented} makes $\Phi$ into a symmetric monoidal functor.
\end{proof}

When $\cC$ is a symmetric \emph{tensor category} \cite[Definition 4.1.1]{EGNO2015tensor}, a 2-d (oriented or unoriented) TQFT $F$ gives a numerical invariant of (oriented or unoriented) surfaces $M$ by regarding them as morphisms from the empty manifold to the empty manifold \cite[Section 2]{atiyah1988tqft}, whence $F(M) \in \End{\cC}{\mathbb{1}} \cong \mathbb{k}$ can be identified with a scalar in $\mathbb{k}$ an algebraically closed field. As a consequence of the (symmetric) monoidal structure of $\msf{SymMonCat}(\msf{2Cob},\cC)$ and $\msf{SymMonCat}(\msf{2UCob},\cC)$, we obtain that given $F$ and $G$ two TQFTs, the invariant of surfaces given by $F\otimes G$ their tensor product as TQFTs is precisely $(F\otimes G)(M) = F(M)\otimes G(M)$, namely $F(M)G(M)$ the (commutative) product in $\mathbb{k}$ of the invariants associated to $F$ and $G$. This reasoning holds in any dimension, whence the invariant given by the tensor product of $n$-dimensional TQFTs is the product of the invariants given by the $n$-dimensional TQFTs. This can be translated via the canonical equivalences of Theorems~\ref{theo:oriented} and \ref{theo:unoriented} into the usual statement saying that the invariant given by the tensor product of Frobenius algebras is the product of the invariants given by each algebra. Our proof differs from the usual one since it does not rely on cutting the surface nor in the existence of a normal form.

Unfortunately, in full generality, other equivalences similar in nature to the ones treated above are not symmetric monoidal equivalences. For example, \cite[Theorem 4.2.7]{windelborn2023hqft} exhibits an equivalence between 2-dimensional homotopy quantum field theories over a topological space $X$ \cite[Definition 3.0.2]{windelborn2023hqft} and twisted Frobenius algebras \cite[Definition 1.2]{staicturaev2010remarks}. These notions can be effortlessly generalized to admit values in $\cC$ and to be certain objects in $G$-graded categories $\bigoplus_{g\in G}{\cC_g}$, respectively. The category $\msf{SymMonCat}(\msf{2Bord}(X),\cC)$ of 2-d HQFTs is always symmetric monoidal (whence the invariant of the product is the product of the invariants, as before). However, already for $\cC = \mathbb{k}-\msf{Vec}_{G}$ the category of twisted Frobenius algebras in $\cC$ does not inherit a monoidal structure because twisted associativity is not preserved.

\section*{Acknowledgments}

We thank the anonymous referee, whose corrections and suggestions helped improve and clarify this manuscript.

%------------------------------------------------------------------------------

\bibliographystyle{alpha}
\bibliography{ref_sym_mon_equ_TQFT_FA}

%------------------------------------------------------------------------------
\end{document}